\documentclass[12pt]{article}
\usepackage[utf8x]{inputenc}
\usepackage{amsthm}
\usepackage{amsmath}
\usepackage{amssymb}
\usepackage{tikz}
\usepackage{enumitem}
\usepackage{mdwlist}
\usepackage{verbatim}
\usepackage[british]{babel}
\usepackage{scalefnt}
\usepackage{tikz,fullpage}
\usepackage{tkz-berge}
\usepackage{varioref}
\usepackage{cite}
\usepackage{caption}
\usepackage[T1]{fontenc}
\usepackage[twoside,a4paper]{geometry}
\usepackage{graphicx}

\title{A note on total and list edge-colouring of graphs of tree-width 3}
\author{Richard Lang\\ Universidad de Chile,\\ 
Santiago, Chile
 \\ \small rlang@dim.uchile.cl }

\theoremstyle{plain}
\numberwithin{equation}{section}
\numberwithin{figure}{section}
\newtheorem{conjecture}{Conjecture}
\newtheorem{theorem}{Theorem}
\newtheorem{lemma}{Lemma}[section]

\theoremstyle{definition}

\newcommand{\ch}{\mathrm{ch}}

\begin{document}


\maketitle
\begin{abstract}
 It is shown that Halin graphs are $\Delta$-edge-choosable and that graphs of tree-width 3 are
 $(\Delta+1)$-edge-choosable and $(\Delta +2)$-total-colourable.
\end{abstract}

\section{Introduction}
\label{sec:introduction}
  In this note we present some results concerning the list chromatic index $\ch'(G)$ and the total chromatic number $\chi''(G)$
  of graphs $G$ of tree-width 3 (see Section~\ref{sec:preliminaries} for proper definitions). One of the central open  questions in the field of list colouring is known as the \emph{list colouring conjecture}: 
 \begin{conjecture}
  \label{con:list-colouring-conjecture}
  For all graphs $G$ it holds that $\ch'(G) = \chi'(G)$.
 \end{conjecture}
  Conjecture~\ref{con:list-colouring-conjecture} appeared for the first time in print in 1985 \cite{boha85}, but was,
 according to Alon~\cite{Alon93}, Woodall~\cite{woodall01} and Jensen and Toft~\cite{JeTo95}, suggested independently by Vizing, 
 Albertson, Collins, Erdös, 
 Tucker and Gupta in the late seventies.
 If Conjecture~\ref{con:list-colouring-conjecture} is true we have $\chi''(G) \leq \Delta(G) +3$ for all graphs $G$\footnote{If we colour the vertices of $G$ using the colours $C = \{ 1, \ldots, \Delta(G)+3\},$ then for each edge there are still $\Delta(G)+1$ colours of $C$ available, which permits a total colouring if Conjecture~\ref{con:list-colouring-conjecture} holds.}.
 The \emph{total colouring conjecture} asserts a little more:
  \begin{conjecture}
  \label{con:total-colouring-conjecture}
  For all graphs $G$ it holds that $\chi''(G) \leq \Delta(G)+2$.
 \end{conjecture}
  Conjecture~\ref{con:total-colouring-conjecture} has been promoted independently by Behzad~\cite{behz63} and Vizing~\cite{viz-total-76}.
 Our first result is a list version of Vizing's theorem for graphs of tree-width 3.

\begin{theorem}
\label{thm:treewidth-leq-3_implies_delta+1-choosable}
 For a graph $G$ of tree-width 3 it holds that $\ch'(G) \leq \Delta(G) +1$.
\end{theorem}
 There are graphs for which this is a sharp bound, see Figure~\ref{fig:counterexample}. Conjecture~\ref{con:total-colouring-conjecture}
 has been proved for graphs of maximum degree at most 5 by Kostochka~\cite{journals/dm/Kostochka96} and 
 Rosenfeld~\cite{raey}. 
 We will use this and a slight variation of the proof of Theorem~\ref{thm:treewidth-leq-3_implies_delta+1-choosable} to show Conjecture~\ref{con:total-colouring-conjecture} for graphs of  tree-width 3 in Section~\ref{sec:total}.
\begin{theorem}
\label{thm:treewidth-leq-3_implies_total_delta+2_colourable}
 For a graph $G$ of tree-width 3 it holds that $\chi''(G) \leq \Delta(G) +2$.
\end{theorem}
 A Halin graph is constructed by taking a planar embedding of a tree without vertices of degree 2 and  connecting all leaves of the tree with a cycle that passes around the tree in the natural cyclic order. Halin graphs  have tree-width 3~\cite{bodl-halin-tw}. In Section~\ref{sec:halin} we prove Conjecture~\ref{con:list-colouring-conjecture} for Halin graphs.
\begin{theorem}
 \label{cor:halin}
 For a Halin graph $G$ it holds that $\ch'(G) = \Delta(G)$.
\end{theorem}
\begin{figure}
\centering
 \tikzstyle{vertex}=[circle,draw,minimum size=8pt,inner sep=0pt]
\tikzstyle{edge} = [draw,-]
\tikzstyle{weight} = [font=\small]
  \begin{tikzpicture}[scale=0.4]

    \foreach \pos/\name in {{(1,1)/v_3}, {(-2,-1)/v_2}, {(4,-1)/v_4},
                            {(1,4)/v_1},{(0,1)/v_5}}
       \node[vertex, align=center] (\name) at \pos {};

       \foreach \source/ \dest in {v_1/v_2, v_1/v_3, v_1/v_4, v_1/v_5, 
                                         v_2/v_3, v_2/v_4, v_2/v_5,
                                         v_3/v_4, v_3/v_5}
       \path[edge] (\source) -- node[weight] {} (\dest);

    \foreach \vertex  in {}
        \path node[selected vertex] at (\vertex) {};
 \end{tikzpicture}
\caption{A graph of tree-width 3 and chromatic index 5.} 
\label{fig:counterexample}
\end{figure}
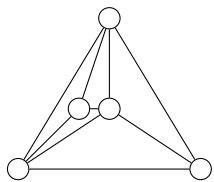
   While still open in general, Conjecture~\ref{con:list-colouring-conjecture} has been verified for  some particular families of graphs. Galvin  
 proved that $\ch'(G) = \Delta(G)$ for all bipartite multigraphs $G$~\cite{journals/jct/Galvin95}.
 Ellingham and Goddyn used a method of Alon and Tarsi to show  that every $d$-regular, $d$-edge-colourable, planar multigraph is $d$-edge-choosable \cite{journals/combinatorica/EllinghamG96}. 
 Some years later Juvan, Mohar and Thomas showed that $\ch'(G) = \Delta(G)$ holds for series parallel graphs $G$~\cite{Juvan99listedge-colorings}. 
 This class of graphs can also be characterized in terms of tree-width. 
 Series parallel graphs have tree-width at most 2. Conversely, 
 a graph has tree-width at most 2 if and only if every biconnected component is series parallel~\cite{Bodlaender1998}.
 This is why we are interested in list 
 edge-colouring graphs of tree-width 3. 
 The methods presented here are extended in~\cite{lang-13-1} in order to prove Conjecture \ref{con:list-colouring-conjecture} for graphs
 of tree-width 3 and a high maximum degree.

 \section{Graphs of tree-width 3 are $(\Delta(G)+1)$-choosable}
 \label{sec:preliminaries}
 We will mostly use standard notation as seen in \cite{Diestel00}. All graphs are finite and simple. 
 The size of a graph $G$ is $|V(G)|+|E(G)|$. $H$ is smaller
 than $G$ if its size is less than the size of $G$. 
 For a graph $G$ a \emph{tree decomposition} $(T,\mathcal{V})$ consists of a tree $T$ and a collection $\mathcal{V} = \{V_t \textit{ ; } t \in V(T) \}$ 
 of \emph{bags} $V_t \subset V(G)$ such that
 \begin{itemize}
  \item $V(G) = \bigcup_{t \in V(T)} V_t,$
  \item for each $vw \in E(G)$ there exists a vertex $t \in V(T)$ such that $v$, $w \in V_t$ and
  \item if $v \in V_{t_1} \cap V_{t_2}$ then $v \in V_t$ for all vertices $t$ that lie on the path connecting $t_1$ and $t_2$ in $T.$
 \end{itemize}
 
 A tree decomposition $(T,\mathcal{V})$ of $G$ has \emph{width} $k$, if all bags have a size of at most $k+1$. 
 The \emph{tree-width} of $G$ is the smallest number $k$ for which there exists a width $k$ tree decomposition of $G$. As
 our proofs are based on minimality it is important to mention that the graphs of tree-width at most $k$ form a minor-closed family. 
 We call a width $k$ tree decomposition $(T,\mathcal{V})$ \emph{smooth} if each bag has size $k+1$, no two
 bags are identical and for each $t_1t_2 \in E(T)$ the bags of $t_1$ and $t_2$ share exactly $k$ vertices. A graph of tree-width $k$ has a smooth
 width $k$ tree decomposition~\cite{Bodlaender1998}.
 Given a tree decomposition $(T,\mathcal{V})$ of $G$ where $T$ is rooted in some vertex $r \in V (T )$ we define the \emph{height} of
 any vertex $t \in V(T)$ to be the distance from $r$ to $t$.

 An \emph{instance of list edge-colouring} consists of a graph $G$ and an \emph{assignment of lists} 
 $L:E(G) \rightarrow \mathcal{P}(\mathbb{N})$ that maps the edges of $G$
  to \emph{lists of colours} $L(e)$. 
 A function $\mathcal{C}:E(G) \rightarrow \mathbb{N}$ is called an
 \emph{$L$-edge-colouring} of $G$, if $\mathcal{C}(e) \in L(e)$ for each $e \in E(G)$ and no two adjacent edges receive the same
 colour. $G$ is said to be \emph{$k$-edge-choosable}, if for each assignment of lists $L$ to the edges of $G$, where all lists have
 a size of at least $k$, there is an $L$-edge-colouring of $G$. The \emph{list chromatic index}, denoted by $\ch'(G)$, is the smallest integer $k$
 for which a graph $G$ is $k$-edge-choosable.

 
 Let $G$ be a graph and $L$ an assignment of lists to the edges of $G$.
 For an $L$-edge-colouring $\mathcal{C}$ of some subgraph $H \subset G$ we call a colour $c$
 of the list of an uncoloured edge $e$ \emph{available}, if no edge adjacent to $e$ has already been coloured with $c$. 
 The set of available colours of $e$ is called the \emph{list of remaining colours} and denoted by $L^{\mathcal{C}}(e)$. 
 We can always try to \emph{colour $G$ greedily}, by iteratively colouring 
 the edge with the smallest list of available colours with an arbitrary available colour. 
 
 For a subset of vertices $W \subset V(G),$ we denote by $G\langle W \rangle$ the graph with vertex set $W \cup N(W)$ and edge set 
 $E(G) \setminus E(G-W)$. 
 Let $G$ be a graph with an assignment of lists  $L$ to the edges of $G$ such that  each list $L(vw)$ has a size of at least $\max(\deg_G(v),\deg_G(w))+1$. Suppose  that for some proper  subset of vertices $W \subset V(G),$ we can find an $L$-edge-colouring $\mathcal{C}$ of the graph $G - W$. In order  to extend $\mathcal{C}$ to an $L$-edge-colouring of $G$ we need to find an $L^\mathcal{C}$-colouring of $G \langle W \rangle$. For an edge $w_1w_2 \in E(G)$ with $w_1,$ $w_2 \in W$ we have 
 \begin{equation}
 \label{equ:remaining-colours-1}
  |L^{\mathcal{C}}(w_1w_2)| = |L(w_1w_2)|.
 \end{equation}
 For  an edge $vw \in E(G)$ with $w \in W$ and $v \in V(G) \setminus W$, we have 
 \begin{equation}
 \label{equ:remaining-colours-2}
  |L^{\mathcal{C}}(vw)|\geq |L(vw)| - \deg_{G-W}(v) \geq \deg_{G\langle W \rangle}(v)+1,
 \end{equation} 
 since $\max(\deg_G(v),\deg_G(w)) \geq \deg_G(v) = \deg_{G-W}(v) + \deg_{G\langle W \rangle}(v).$
 In the proofs of the following results, we will generally assume that the size of each list is exactly the size of its given lower bound. 
 The next theorem has already been mentioned in the introduction.

\begin{theorem}[Galvin, 1994]
 \label{thm:galvin}
 Let $G$ be a bipartite graph; then $\ch'(G) = \Delta(G).$ 
 \end{theorem}
  
 \label{sec:delta+1}
\begin{lemma}
\label{lem:ballon}
 Let $G$ be a cycle $e_1$, $\ldots$, $e_n$ with an additional edge $f$ that is incident exactly to the one vertex of $C$ that
 $e_1$ and $e_n$ share. For any
 assignment of lists $L$, where each list has a a size of at least 2 for all edges and the size of the list of $e_1$ is
 at least 3, there is an $L$-edge-colouring.
\end{lemma}
\begin{proof}
 If there is a colour $c \in L({e_n}) \setminus L({f})$, colour $e_n$ with $c$ and finish greedily. This yields $L(f) = L({e_n})$ and
 so there is a colour $c \in L({e_1}) \setminus (L(f) \cup L({e_n}))$. Colour $e_1$ with $c$ and finish greedily.
 \end{proof}

 \begin{figure}
\centering
\tikzstyle{vertex}=[circle,draw,minimum size=14pt,inner sep=0pt]
\tikzstyle{edge} = [draw,-]
\tikzstyle{weight} = [font=\small,draw,fill           = white,
                                  text           = black]
\begin{tikzpicture}[scale=1.2]
    \foreach \pos/\name in {{(1,1)/w_1}, {(-1,1)/w_0}, {(0,0)/w_2},{(0,2)/w_3},{(2,1)/w_4}}
     \node[vertex, align=center] (\name) at \pos {$\name$};

    \foreach \source/ \dest /\weight in {w_0/w_1/5, w_0/w_2/3, w_0/w_3/3, w_1/w_2/3, w_1/w_3/3, w_1/w_4/2}
       \path[edge] (\source) -- node[weight] {$\weight$} (\dest);
 \end{tikzpicture}
\caption{The integers on the edges indicate the minimum sizes of the respective lists.}
 \label{fig:treewidth-leq-3_implies_delta+1-choosable}
\end{figure}

The next lemma implies Theorem \ref{thm:treewidth-leq-3_implies_delta+1-choosable}.
\begin{lemma}
\label{lem:treewidth-leq-3_implies_delta+1-choosable}
  Let $G$ be a graph of tree-width at most 3 with an assignment of lists $L$ such that each list $L(vw)$ has a size of at least
  ${\max(\deg(v),\deg(w))}{+1}$ for each edge $vw \in E(G)$. Then there is an $L$-edge-colouring of $G$.
\end{lemma}
\begin{proof}
 We will assume that the lemma is wrong and obtain a contradiction. Let $G$ be a smallest counterexample
 to the lemma with an assignment of lists  $L$ to the edges of $G$
 for which each $L(vw)$ has a size of at least
 $\max(\deg(v),\deg(w))+1$ for each edge $vw \in E(G)$ such that there is no $L$-edge-colouring of $G$. 
 
 If there is a vertex $v \in V(G)$ of degree at most 2, we can find an $L$-edge-colouring of $G-v$ by minimality and extend this to an
 $L$-edge-colouring of $G$ by colouring the edges adjacent to $v$ greedily from the lists of remaining colours. Therefore we might assume that 
$G$ has a minimum degree of at least 3.  This implies that $G$ has tree-width 3, since the bag of any leaf of a width 2 tree-decomposition contains a vertex of degree at most 2.  Let $(T,\mathcal{V})$ be a smooth width 3 tree decomposition of $G$ where $T$ is rooted in some arbitrary vertex. Let $t \in V(T)$ be a vertex of degree at least 2 in $T$ and of maximum height.
 
 Suppose the neighbourhood of $t$ contains at least two leaves $t_1,~t_2 \in V(T)$. Since $(T,\mathcal{V})$ is smooth, there are vertices $v_1 \in V_{t_1}$ and $v_2\in V_{t_2}$ which are uniquely in $V_{t_1}$ and $V_{t_2}$. 
 We have $\deg(v_1) = \deg(v_2) =3,$ $v_1$ and $v_2$ are not adjacent and $|N(v_1) \cup N(v_2)| \leq 4$. By minimality we can find an $L$-edge-colouring of the graph $G-v_1-v_2$. If $v_1$ and $v_2$ have the same neighbourhood we can extend this to an $L$-edge-colouring of $G$  by applying Theorem~\ref{thm:galvin} to the bipartite graph $G\langle \{v_1, v_2\} \rangle,$ as their lists of remaining colours have each a size of at least 3  by~(\ref{equ:remaining-colours-2}).  If $v_1$ and $v_2$ have distinct neighbourhoods,  colour the two edges in $G\langle \{v_1, v_2\} \rangle$ that are adjacent to vertices outside of $N(v_1) \cap N(v_2)$ greedily. We can
 extend this to an $L$-edge-colouring of $G$ by applying Theorem \ref{thm:galvin} as the uncoloured edges form a 4-cycle and their lists of remaining colours have each at least size 2 by~(\ref{equ:remaining-colours-2}). 
 
 So we can assume that $t$ is adjacent to exactly one leaf $t_0 \in V(T)$. As $(T,\mathcal{V})$ is smooth and $t$ has maximum height, there are vertices $w_0 \in V_{t_0}$ and $w_1 \in V_{t} \cap V_{t_0}$ that appear uniquely in $V_{t_0}$ and $V_{t} \cap V_{t_0}$. Let  $V_t = \{w_1,w_2,w_3,w_4\}$. Since $w_0$ appears uniquely in the bag of the leaf $t_0,$ it has a degree of at most $3$. So $\deg(w_0) = 3$ and we may also assume that $w_0w_4 \notin E(G)$. Further, since the neighbours of $w_1$ are either in $V_{t_0}$ or $V_{t},$ $w_1$ has a degree of at most 4. 
 
 Now if $\deg(w_1) = 3$, we can choose an $L$-edge-colouring of $G-\{w_0,w_1\}$ by minimality and extend this to an $L$-edge-colouring of $G$ as follows. If $w_1w_4$ is an edge, colour it with an arbitrary colour and apply Lemma~\ref{lem:ballon} to $G\langle \{w_0, w_1\} \rangle-w_0w_4$. So $N(w_1) = \{w_0, w_2, v_1\}$ and we can colour $w_0w_1$ with an arbitrary colour and apply Theorem~\ref{thm:galvin} to the 4-cycle $G\langle \{w_0, w_1\} \rangle-w_0w_1$. This yields $\deg(w_1) = 4$. 
 
 We pick a final $L$-edge-colouring of  $G-w_0-w_1$ by minimality. In order to extend this to an $L$-edge-colouring of $G$ we need to colour the edges of the graph $G\langle \{w_0, w_1\} \rangle$  from the lists of remaining colours. This instance of list edge-colouring is shown
 in Figure \ref{fig:treewidth-leq-3_implies_delta+1-choosable}. The lower bounds on the lists of remaining colours are given by~(\ref{equ:remaining-colours-1})
 and~(\ref{equ:remaining-colours-2}). Colour $w_1w_2$ with some colour 
 $c \in L(w_1w_2) \setminus L(w_1w_4),$ colour the edge $w_2w_0$ greedily and finish as shown in Lemma \ref{lem:ballon}. A contradiction.
 \end{proof}

We are now ready to proof Theorem~\ref{thm:treewidth-leq-3_implies_delta+1-choosable}.

\begin{proof}[of Theorem~\ref{thm:treewidth-leq-3_implies_delta+1-choosable}]
Let $L$ be an assignment of lists to the edges of $G$ such that each list has a size of at least $\Delta(G) +1$. Since for each edge $vw \in E(G)$ we have
${\max(\deg(v),\deg(w))}{+1} \leq \Delta(G)+1,$ there is an $L$-edge-colouring of $G$ by Lemma~\ref{lem:treewidth-leq-3_implies_delta+1-choosable}.
 \end{proof}

\section{Graphs of tree-width 3 are $(\Delta +2)$-total-colourable}
\label{sec:total}
 As mentioned in the introduction Conjecture \ref{con:total-colouring-conjecture} is true for graphs of maximum degree
 at most 5. In this section we handle the case where a graph has tree-width 3, and maximum degree greater than 5 and use this to proof 
 Theorem~\ref{thm:treewidth-leq-3_implies_total_delta+2_colourable}.

 An \emph{instance of list total colouring} consists of a graph $G$ and an assignment of lists
 $L:V(G)\cup E(G) \rightarrow \mathcal{P}(\mathbb{N})$ to the vertices and edges of $G$. 
 A function $\mathcal{C}:V(G)\cup E(G) \rightarrow \mathbb{N}$ is called an
 \emph{$L$-total-colouring} of $G$, if $\mathcal{C}(v) \in L(v)$ for each $v \in V(G)$, 
 $\mathcal{C}(e) \in L(e)$ for each $e \in E(G)$, no two adjacent vertices receive the same colour, no
 two adjacant edges receive the same
 colour and no edge has the same colour as one of its ends. 
 $G$ is said to be \emph{$k$-total-choosable}, if for each assignment of lists $L$ to the vertices and edges of $G$, 
 where all lists have a size of at least $k$, there is an $L$-total-colouring of $G$. 
 
  Let $G$ be a graph and $L$ an assignment of lists to the edges and vertices of $G$.
 For an $L$-total-colouring $\mathcal{C}$ of some subgraph $H \subset G$ we call a colour $c$
 of the list of an uncoloured edge  $e$ \emph{available}, if no edge adjacent to $e$ and no endvertex of $e$ has already been coloured with $c$. 
 Similarly we call a colour $c$ of the lists of an uncoloured vertex $v$ \emph{available}, if none of the edges and vertices adjacent to $v$ have already
 been coloured with $c$.
 The set of available colours of an edge or vertex $x$ is called \emph{list of remaining colours} and denoted by $L^{\mathcal{C}}(x)$. 
 We can always try to \emph{colour $G$ greedily}, by iteratively colouring 
 the edge or vertex with the smallest list of available colours with an arbitrarily available colour.
 
 For $\Delta \geq 1$, let $G$ be a graph with an assignment of lists  $L$ to the edges and vertices of $G$ such that the list of each edge $vw$ and each vertex $w$ has a size of at  least $\Delta+2$. Suppose  that for some proper  subset of vertices $W \subset V(G),$ we can find an $L$-total-colouring $\mathcal{C}$ of the graph $G - W$. In order to extend $\mathcal{C}$ to an $L$-total-colouring of $G$ we need to find an $L^\mathcal{C}$-colouring of $G \langle W \rangle$. 
 Since $\Delta \geq \deg_G(v) = \deg_{G-W}(v) + \deg_{G\langle W \rangle}(v),$ we have for an edge $vw \in E(G)$ with $w \in W$ and $v \in V(G) \setminus W$
 \begin{equation}
 \label{equ:remaining-colours-total-2}
  |L^{\mathcal{C}}(vw)|\geq |L(vw)| - \deg_{G-W}(v) -1 \geq \deg_{G\langle W \rangle}(v)+1.
 \end{equation} 
 For any vertex $w \in W$ we have
 \begin{equation}
 \label{equ:remaining-colours-total-3}
  |L^{\mathcal{C}}(w)|\geq |L(w)| - |N(w)  \setminus W|.
 \end{equation} 
  The proof of the following lemma mirrors the proof of Lemma \ref{lem:treewidth-leq-3_implies_delta+1-choosable}. 

  \begin{lemma}
\label{lem:treewidth-leq-3_and_max_deg_geq_6_implies_delta+2-total-choosable}
  Let $G$ be a graph of tree-width at most 3 
  with an assignment of lists $L$ to the vertices and edges of $G$ such that
  each list has a size of at least $\max(5,\Delta(G)) +2$. Then $G$ has an $L$-total-colouring.
\end{lemma}

  \begin{proof}
 We will assume the lemma is wrong and obtain a contradiction.
 Let $G$ be a smallest counterexample
 to the lemma with an assignment of lists $L$ to the vertices and edges of $G$ 
 for which each list has a size of at least $\Delta:= \max(5,\Delta(G))+2$ such that there is no $L$-total-colouring of $G$.

 If there is a vertex $v$ of degree at most 2, we can find a $L$-total-colouring of $G-v$ by minimality. By~(\ref{equ:remaining-colours-total-2}) and 
 (\ref{equ:remaining-colours-total-3}) the lists of remaining colours of $G\langle \{v\} \rangle$ retain sizes large enough to extend this colouring greedily. 
 Therefore we might assume that $G$ has a minimum degree of at least 3. As before it follows that $G$ has tree-width 3. 
 Let $(T,\mathcal{V})$ be a smooth width 3 tree decomposition of $G$ where $T$ is rooted in some arbitrary vertex and let
 $t \in V(T)$ be a vertex of degree at least 2 in $T$ of maximum height.

 Suppose the neighbourhood of $t$
 contains at least two leaves $t_1,$ $t_2 \in V(T)$. Since $(T,\mathcal{V})$ is smooth and $t$ has maximum height, there are vertices $v_1 \in V_{t_1}$ and $v_2 \in V_{t_2}$ 
  that are uniquely in $V_{t_1}$ respectively $V_{t_2}$. 
 We have $\deg(v_1) = \deg(v_2) =3,$ $v_1$ and $v_2$ are not adjacent and $|N(v_1) \cup N(v_2)| \leq 4$.
 By minimality we can find an $L$-total-colouring of the graph
 $G-v_1-v_2$. 
 If $v_1$ and $v_2$ have the same neighbourhood we first apply Theorem~\ref{thm:galvin} to the bipartite graph induced by the 
 edges adjacent to $v_1$ and $v_2$. This is possible since the lists of remaining colours have sizes of at least 3 by~(\ref{equ:remaining-colours-total-2}).
 The lists of the vertices $v_1$ and $v_2$ retain at least one available colour by~(\ref{equ:remaining-colours-total-3}). So we can finish colouring 
 greedily to extend this to an $L$-edge-colouring of $G.$ 
  If $v_1$ and $v_2$ have distinct neighbourhoods, 
 colour the two uncoloured edges adjacent to vertices outside of $N(v_1) \cap N(v_2)$ greedily. We apply 
 Theorem \ref{thm:galvin} to the bipartite graph induced by the 
 edges between $\{v_1, v_2\}$ and $N(v_1) \cap N(v_2),$ which is possible because their lists of remaining colours have each a size of at least 2 
 by~(\ref{equ:remaining-colours-total-2}). As before colour the vertices $v_1$ and $v_2$ greedily to extend this to an $L$-edge-colouring of $G.$
 
 Thus we can assume that $t$ is incident to exactly one leaf $t_0 \in V(T)$. Since $(T,\mathcal{V})$ is smooth, there are vertices $w_0\in V_{t_0}$ and 
 $w_1 \in V_{t} \cap V_{t_0}$
  that appear uniquely in $V_{t_0}$ respectively $V_{t} \cap V_{t_0}$. By minimality there is 
 an $L$-total-colouring $\mathcal{C}$ of $G-w_0w_1$. Delete the colour of $w_0$ from $\mathcal{C}$. As $\deg(w_0) = 3$ and
 $\deg(w_1) \leq 4$ we can finish greedily. Contradiction.
 \end{proof}

\section{Halin graphs are $\Delta$-edge-choosable}
\label{sec:halin}
 
 To show the case where $\Delta(G) = 3$ of Theorem~\ref{cor:halin} we will use a
 result of Ellingham and Goddyn~\cite{journals/combinatorica/EllinghamG96}.
 \begin{theorem}
 \label{thm:ellignham}
  Let $G$ be a $d$-regular planar graph. If $G$ is $d$-edge-colourable, then $G$ is $d$-edge-choosable.
 \end{theorem}
 
 The next Lemma is a corollary to the 4 Colour Theorem, which is equivalent to the statement that every bridgeless cubic planar graph chromatic index 3. We include it for sake of completeness. 
 
 \begin{lemma}
 \label{lem:halin-3-edge-colourable}
  Let $G$ be a 3-regular Halin graph. Then $G$ is 3-edge-colourable.
 \end{lemma}
 
\begin{proof}
  We will assume that the lemma is wrong and obtain a contradiction. Consider a smallest counterexample $G$ to the
 lemma with a tree $T$ and a cycle $C$ that passes along the leaves of $T$. The cycle 
 $C$ has at least 4 vertices, since otherwise $G$ is the complete graph on 4 vertices and we are done. 
 Choose an arbitrary vertex to be the root of $T$ and let $v$ be a vertex of maximum height among the vertices  of degree 3 in $T$. 
 By maximality $v$ has exactly two neighbours $v_1$ and $v_2$ that lie on the cycle $C$. As $C$ has at least 4 vertices, there are distinct vertices $v_0,$ $v_3 \in V(C)$ to which $v_1$ respectively $v_2$ are 
 adjacent. Let $G_1$  be the graph obtained from $G - v_1$ by adding the edge $v_0v_2$ and $v_2w$, where $w$ is the third neighbour of $v$. $G_1$ is a 3-regular Halin graph smaller than $G$. So by minimality there is  an edge-colouring $\mathcal{C}_1$ of $G_1$ using the colours $\{1, 2, 3\}.$ We can extract an edge-colouring of $G-v_1-v_2-v$ from the edge-colouring of $G_1$ and extend this greedily to an edge colouring of $G$ using only the colours $\{1, 2, 3\}.$ A contradiction. 
 \end{proof}

 For a graph $G$ with an assignment of lists $L$ to the edges of $G$ and ${e,~f \in E(G)}$ we call two colours $c_1 \in L(e)$
 and $c_2 \in L(f)$ \emph{compatible} if $c_1 = c_2$ or if for each edge $g$ that is adjacent to both $e$ and $f$ the list $L(g)$ contains
 at most one of the two colours $c_1$ and $c_2$.
 The following lemma turns out to be quite useful in order solve instance of list edge-colourings of small graphs. The idea for the proof can be extracted from 
 \cite{Cariolaro_theedge-choosability}.
 \begin{lemma}
 \label{lem:cariolaro}
 Let $G$ be a graph with an assignment of lists $L$ to the edges of $G$ 
 and $v_1v_2,$ $w_1w_2 \in E(G)$ two edges that are not adjacent. If it holds that
 $$|L({v_1v_2})||L({w_1w_2})| > \sum_{v_iw_j \in E(G)} \lfloor\frac{|L({v_iw_j})|}{2}\rfloor  \lceil\frac{|L({v_iw_j})|}{2}\rceil$$ 
 then there are
 two compatible colours $c_1 \in L({v_1v_2})$ and $c_2 \in L({w_1w_2})$.
\end{lemma}
 \begin{proof}
  If the lists
  of the edges $v_1v_2$ and $w_1w_2$ share a colour $c$ we are done.
  Therefore assume that $L({v_1v_2}) \cap L({w_1w_2}) = \emptyset$. This yields that there are  $|L({v_1v_2})||L({w_1w_2})|$ pairs of
  distinct colours $(c_1,c_2)$ with $c_1 \in L({v_1v_2})$ and $c_2 \in L({w_1w_2})$. But an edge $v_iw_j \in E(G)$ 
  for $1\leq i,j\leq 2$ can contain both colours of at most
  $\lfloor\frac{|L({v_iw_j})|}{2}\rfloor  \lceil\frac{|L({v_iw_j})|}{2}\rceil$ of those pairs.
  So if the above inequation holds we can find the desired two compatible colours
  $c_1 \in L({v_1v_2})$ and $c_2 \in L({w_1w_2})$.
  \end{proof}
 Remark that the inequality holds if all involved lists have a size of exactly $k,$ where $k$ is an odd number. The next
 lemma implies that Halin graphs of maximum degree $\Delta \geq 4$ are $\Delta$-edge-choosable.

\begin{figure}
\centering
\tikzstyle{vertex}=[circle,draw,minimum size=14pt,inner sep=0pt]
\tikzstyle{edge} = [draw,-]
\tikzstyle{weight} = [font=\small,draw,fill           = white,
                                  text           = black]
\begin{tikzpicture}[scale=1]
    \foreach \pos/\name in {{(1,2)/w_1}, {(2,3)/w_2},{(0,1)/w_0},
                            {(3,2)/w_3},{(2,0)/v}}
     \node[vertex, align=center] (\name) at \pos {$\name$};
    \foreach \pos/\name in {{(4,1)/w_4}}
     \node[vertex, align=center] (\name) at \pos {};

    \foreach \source/ \dest /\weight in {w_0/w_1/2, w_1/w_2/4, w_2/w_3/4, w_1/v/3, w_2/v/3, w_3/v/3, w_3/w_4/2}
       \path[edge] (\source) -- node[weight] {$\weight$} (\dest);
 \end{tikzpicture}
       \caption{The integers on the edges indicate the minimum sizes of the respective lists.}
         \label{fig:halin-1}
 \end{figure}
 \begin{figure}\centering
 \centering
\tikzstyle{vertex}=[circle,draw,minimum size=14pt,inner sep=0pt]
\tikzstyle{edge} = [draw,-]
\tikzstyle{weight} = [font=\small,draw,fill           = white,
                                  text           = black]
 \begin{tikzpicture}[scale=1]
    \foreach \pos/\name in {{(0,1)/v_0}, {(0.8,2)/v_1}, {(2.2,2)/v_2},
                            {(3,1)/v_3},{(1.5,0)/v}}
       \node[vertex, align=center] (\name) at \pos {$\name$};
    \foreach \source/ \dest /\weight in {v_0/v_1/2, v_1/v_2/4, v_2/v_3/2, v_1/v/3, v_2/v/3}
       \path[edge] (\source) -- node[weight] {$\weight$} (\dest);  
 \end{tikzpicture}
      \caption{The integers on the edges indicate the minimum sizes of the respective lists.}
        \label{fig:halin-2}
\end{figure}
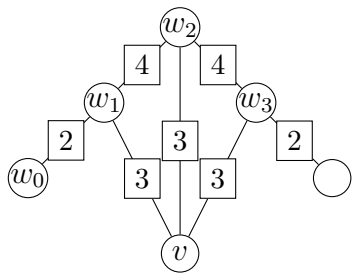

\begin{lemma}
\label{lem:halin}
 Let $G$ be a Halin graph with an assignment of lists $L$ to the edges of $G$ such that
 for each edge $vw \in E(G)$ the list $L(vw)$ has a size of at least 
  $\max(\deg(v) , \deg(w),4).$ Then there is an $L$-edge-colouring of $G$.
\end{lemma}
\begin{proof}
 We will assume that the lemma is wrong and obtain a contradiction. Consider a smallest counterexample $G$ to the
 lemma with a tree $T$ and a cycle $C$ that passes along the leaves of $T$. Let
  $L(vw)$ be an assignment of lists to the edges $ vw \in E(G)$, where each list has a size of at least $\max(\deg(v),  \deg(w),4)$ 
 such that there is no $L$-edge-colouring of $G$.
 We can assume that $G$ is not a complete graph on four vertices by Theorem~\ref{thm:ellignham}
 and hence $C$ has at least 4 vertices. 
 Choose an arbitrary vertex to be the root of $T$ and let $v$ be a vertex of maximum height among the vertices 
 of degree at least 3 in $T$. 
 By maximality $v$ has $\deg(v)-1$ neighbours that lie on the cycle $C$. 
 
 If $\deg(v) \geq 4$ then $v$ has three neighbours $w_1,$ $w_2$ and $w_3 \in V(C)$ with say $N(w_2) = \{v,w_1,w_3\}$ and $N(w_1)= \{v,w_0, w_2\}$. We denote by
 $G_1$ the graph obtained from $G - w_1$ by adding the edge $w_0w_2$.
  Observe that $G_1$ is a Halin graph,
 smaller than $G$ and the degrees of $w_0,$ $w_2$ and $v$ did not increase. Hence there is an $L$-edge-colouring $\mathcal{C}_1$ of 
 $G_1$ by minimality (where an arbitrary list of size 4 was assigned to $w_0w_2$). We can extract an $L$-colouring of the graph $G-w_1-w_2-w_3$ 
 from $\mathcal{C}_1$. In order to extend
 this to an $L$-edge-colouring of $G$ we need colour the graph $G\langle \{w_1,w_2,w_3\} \rangle$ from the lists of remaining colours. The lower bounds on
 the sizes of these lists  are given by~(\ref{equ:remaining-colours-1}) and~(\ref{equ:remaining-colours-2}).
 This instance of list edge-colouring is shown in Figure \ref{fig:halin-1}. 
 By Lemma \ref{lem:cariolaro} we can find two compatible colours $c_1 \in L(vw_1),$ $c_2 \in L(w_2w_3)$. Colour the respective
 edges with $c_1$ and $c_2$ and
 finish greedily. 
 
 Thus we might assume that $\deg(v) = 3$ and so $v$ has exactly two neighbours $v_1$ and $v_2 \in V(C)$. As
 $C$ has at least 4 vertices, there are distinct vertices $v_0,$ $v_3 \in V(C)$ to which $v_1$ respectively $v_2$ are adjacent. 
 Further, $v$ is adjacent to neither $v_0$ nor $v_2$. Let $G_2$  be the graph obtained from $G - v_1-v$ by adding the edge $v_0v_2$ and $v_2w$, where $w$ is the third neighbour of $v$. As before, $G_2$ is still a Halin
 graph, smaller than $G$ and the degree of $v_0$ and $v_2$ did not increase. So by minimality there is  an $L$-edge-colouring $\mathcal{C}_2$ of $G_2$ (where arbitrary lists of size $\max(4,\deg(w))$ was assigned to $v_0v_2$ and $v_2w$). We can extend $\mathcal{C}_2$ to an $L$-edge-colouring of $G$ by colouring the graph $G\langle \{v_1,v_2\} \rangle$ from the lists of remaining colours, which lower bounds
 are given by~(\ref{equ:remaining-colours-1}) and~(\ref{equ:remaining-colours-2}).  Colour the edge between $v$ and the vertex distinct from $v_1$ and $v_2$ greedily.
 The remaining instance of list edge-colouring is shown in Figure \ref{fig:halin-2}. Colour the edge $v_0v_1$ greedily and apply
 Lemma \ref{lem:ballon} to the rest. A contradiction. 
 \end{proof}
\begin{figure}
\centering
 \tikzstyle{vertex}=[circle,draw,minimum size=8pt,inner sep=0pt]
\tikzstyle{edge} = [draw,-]
\tikzstyle{weight} = [font=\small]

 \begin{tikzpicture}[scale=0.5]
    \foreach \pos/\name in {{(0,0)/v_0}, {(0,1)/v_1}, {(1,0)/v_2}, {(0,-1)/v_3}, {(-1,0)/v_4},
                             {(-1,2)/v_{11}}, {(1,2)/v_{12}}, {(2,1)/v_{21}}, {(2,-1)/v_{22}},
                             {(1,-2)/v_{31}}, {(-1,-2)/v_{32}}, {(-2,-1)/v_{41}}, {(-2,1)/v_{42}}}
       \node[vertex, align=center] (\name) at \pos {};
    \foreach \source/ \dest /\weight in {v_0/v_1/{abcd}, v_0/v_2/{a,b,c,d}, v_0/v_3/{a,b,c,d}, v_0/v_4/{a,b,c,d},
                                         v_1/v_{11}/x, v_1/v_{12}/x,
                                         v_2/v_{21}/x, v_2/v_{22}/x,
 v_3/v_{31}/x, v_3/v_{32}/x,
v_4/v_{41}/x, v_4/v_{42}/x,
v_{11}/v_{12}/x,v_{12}/v_{21}/x,v_{21}/v_{22}/x,v_{22}/v_{31}/x,v_{31}/v_{32}/x,v_{32}/v_{41}/x,v_{41}/v_{42}/x,v_{42}/v_{11}/x}
       \path[edge] (\source) -- node[weight] {} (\dest);
 \end{tikzpicture}
\caption{}
\label{fig:halin-bad}
\end{figure}

 Note that the lower bound of 4 on the list size in this result is necessary. For the graph shown in Figure \ref{fig:halin-bad} 
 the lists $L(vw) =  \{1, \ldots , \max(\deg(v),\deg(w))\}$ for each edge $vw$ do not permit an $L$-edge-colouring.

 \begin{proof}[Proof of Theorem~\ref{cor:halin}]
  Let $L$ be an assignment of lists to the edges of $G$ such that each list has a size of at least $\Delta(G)$. If $\Delta(G) \geq 4$, there is an $L$-colouring
  of $G$ by Lemma~\ref{lem:halin}. Otherwise $G$ is planar, 3-regular and can be 3-edge-coloured by Lemma~\ref{lem:halin-3-edge-colourable}. 
  Hence there is an $L$-colouring of $G$ by Theorem~\ref{thm:ellignham}.  
  \end{proof}

 \section*{Acknowledgements}
The author would like to thank Henning Bruhn-Fujimoto for listening and helpful remarks.

\end{document}